\documentclass[11pt]{amsart}
\usepackage{amsmath}
\usepackage{amssymb}
\usepackage{amsfonts}
\usepackage{amsthm}
\usepackage{xcolor}
\usepackage{mathrsfs}
\usepackage{verbatim}
\usepackage{float}
\usepackage{graphicx}
\usepackage{tikz-cd}
\usepackage{marvosym}
\usepackage[hidelinks]{hyperref}
\usepackage{ytableau}
\usepackage{stmaryrd}
\usepackage{multicol}
\usepackage{dynkin-diagrams}
\usepackage[shortlabels]{enumitem}

\usepackage[backend=biber,style=alphabetic]{biblatex}
\usepackage{aliascnt}
\usepackage[hidelinks]{hyperref}
\usepackage[capitalize]{cleveref}

\let\oldtheorem\newtheorem
\RenewDocumentCommand{\newtheorem}{s m o m O{}}{%
\IfBooleanTF{#1}%
{\oldtheorem{#2}{#4}}%
{\IfNoValueTF{#3}{\oldtheorem{#2}{#4}[#5]}%
{\newaliascnt{#2}{#3}%
\oldtheorem{#2}[#2]{#4}%
\aliascntresetthe{#2}}}}

\DeclareLabelalphaTemplate{
  \labelelement{
    \field[uppercase,final]{shorthand}
    \field[uppercase]{label}
    \field[strwidth=3,strside=left,ifnames=1]{labelname}
    \field[uppercase,strwidth=1,strside=left]{labelname}
  }
  \labelelement{
    \field[strwidth=4,strside=right]{year}
  }
}

\AtNextBibliography{\small}
\usepackage{fullpage}  
\usepackage{cleveref}
\usetikzlibrary{positioning}

\newcommand{\RR}{\mathbb{R}}
\newcommand{\ZZ}{\mathbb{Z}}
\newcommand{\PP}{\mathbb{P}}
\newcommand{\A}{\mathbb{A}}
\newcommand{\CC}{\mathbb{C}}

\newcommand{\HH}{\mathbb{H}}
\newcommand{\QQ}{\mathbb{Q}}

\newcommand{\FF}{\mathbb{F}}

\DeclareMathOperator{\Bri}{Bri}
\DeclareMathOperator{\Kle}{Kle}
\DeclareMathOperator{\Ham}{Ham}

\DeclareMathOperator{\Mod}{Mod}

\DeclareMathOperator{\Aut}{Aut}

\DeclareMathOperator{\PGL}{PGL}
\DeclareMathOperator{\PSL}{PSL}
\DeclareMathOperator{\SL}{SL}

\DeclareMathOperator{\Hom}{Hom}

\DeclareMathOperator{\Spec}{Spec}

\DeclareMathOperator{\Pic}{Pic}
\DeclareMathOperator{\Poly}{Poly}

\newtheorem{theorem}{Theorem}[section]
\newtheorem{proposition}[theorem]{Proposition}
\newtheorem{lemma}[theorem]{Lemma}
\newtheorem{corollary}[theorem]{Corollary}

\theoremstyle{definition}
\newtheorem{definition}[theorem]{Definition}
\newtheorem{problem}[theorem]{Problem}

\newtheorem{example}[theorem]{Example}
\newtheorem{remark}[theorem]{Remark}

\numberwithin{equation}{subsection}

\title{The geometry of solutions to the general sextic}
\author{Sidhanth Raman}
\newcommand{\Addresses}{{
  \bigskip
  \footnotesize
  \textsc{Department of Mathematics, University of Chicago}\par\nopagebreak
  \textit{E-mail address}:  \texttt{svraman@uchicago.edu} 
}}

\begin{document}
\begin{abstract}
    We study the geometry of the three simplest known formulas for roots of a general sextic polynomial. In contrast with work of Green on the quintic, we show that no ``simultaneous uniformization" of the normal forms associated to the solutions exists. This demonstrates an intrinsic qualitative complexity of sextic equations that is not present in quintic equations. Additionally, we give a formula for the roots of a general sextic in terms of Hilbert modular cusp forms and the uniformizing differential equation of the associated Hilbert modular surface.
\end{abstract}
\maketitle

\section{Introduction}

In this article, we study the geometry of three projective surfaces endowed with a faithful $A_6$-action. Our interest in these varieties stems from their role in the 19th century theory of formulas: they are central to the three simplest known algebraic formulas for the roots of a general sextic polynomial. Let $\sigma_k$ denote the $k$th elementary symmetric polynomial in $6$ variables. The projective surfaces are defined as follows:

\begin{enumerate}
    \item The \emph{Hamilton surface} is the complete intersection $V_{\Ham} = Z(\sigma_1,\sigma_2,\sigma_3) \subset\PP^5$.
    \item The \emph{Klein projective plane} is $V_{\Kle} = \PP^2$.
    \item The \emph{modular surface} is the complete intersection $V_{\Mod} = Z(\sigma_1,\sigma_2,\sigma_4) \subset \PP^5$.\footnote{This surface earns its name because $V_{\Mod}$ is birational to a certain Hilbert modular surface with level structure. We will describe this space in the subsequent sections.}
\end{enumerate}

In \Cref{Section3}, we will see that $V_{\Ham}$ is a K3 surface and that $V_{\Mod}$ is a surface of general type. Since they are defined using symmetric polynomials, the Hamilton and modular surfaces come equipped with a natural $A_6$-action given by permuting coordinates on each surface.\footnote{In fact, the modular and Hamilton surfaces come equipped with a natural $S_6$-action.} The $A_6$-action on $\PP^2$, called the \emph{Valentiner action}, is more subtle to define --- we give explicit generators for this action in \Cref{KleinsApproach}. One of our main results is concerned with the existence of $A_6$-equivariant maps between these surfaces: 

\begin{theorem}[Nonexistence of equivariant maps]\label{MainTheoremOnMaps}
    There are no $A_6$-equivariant generically finite dominant rational maps between any of the projective $A_6$-surfaces $V_{\Ham}$, $V_{\Kle}$, or $V_{\Mod}$.
\end{theorem}
\begin{remark}
    The nonexistence of any dominant rational map $\PP^2 \to V_{\Ham}$, $\PP^2 \to V_{\Mod}$, or $V_{\Ham} \to V_{\Mod}$ is elementary to see: Kodaira dimension $\kappa$ is non-increasing under dominant morphisms between surfaces, and the results of \Cref{Section3} imply that $\kappa(V_{\Mod}) = 2$, $\kappa(V_{\Ham}) = 0$, and $\kappa(V_{\Kle}) = -\infty$. To prove \Cref{MainTheoremOnMaps}, it suffices to obstruct the existence of equivariant maps between these surfaces when the Kodaira dimension drops.
\end{remark}

To prove \Cref{MainTheoremOnMaps}, we will apply tools from Hodge theory, $A_6$-representation theory, and equivariant birational geometry. When we interpreted this statement in the context of formulas for the roots of the general sextic, we found \Cref{MainTheoremOnMaps} rather surprising. To provide further context and motivation, we now discuss previous work of Green on the quintic \cite{green1978analytic} which inspired this article.

\subsection{Formulas for the roots of a quintic}
The Abel--Ruffini theorem shows that the general quintic polynomial cannot be solved by radicals. However, there are two known formulas for the roots of the general quintic in algebraic functions of one variable: Bring solvably reduced a general quintic to a single variable form using Tschirnhaus transformations \cite{bring1786,jerrard1858essay}, and Klein leveraged the icosahedral action of $A_5$ on $\PP^1$ to reduce the general quintic to a form which could be solved using the single variable icosahedral function \cite{klein1884vorlesungen,klein1958lectures}. Associated to each formula is its \emph{normal form} --- we postpone formalizing the notions of a formula and its normal form until Definitions \ref{formulaDef} and \ref{NormalFormDefinition}. One can informally think of a normal form as the variety whose closed points are in bijective correspondence with sets of ordered roots obtained using the chosen formula.

Klein found, and Green reproduced, a $3$-sheeted $A_5$-equivariant dominant rational map from the Bring normal form $V_{\Bri}$ (a genus $4$ curve) to the Klein normal form (the projective line $\PP^1$), thereby inducing a map from Bring's solution to Klein's solution \cite{green1978analytic}. Moreover, Green uniformized both $V_{\Bri}$ and $\PP^1$ by the upper half plane $\HH^2$, recovering Hermite's analytic solution to the quintic by elliptic modular functions \cite{hermite1858sulla}. Concretely, this $3$-sheeted $A_5$-equivariant map found by Klein and Green gives an algebraic transformation that converts Bring's formula for the roots of a quintic into Klein's formula.

\subsection{Formulas for the roots of a sextic}

Moving up one degree, there are three known solutions to the general sextic polynomial:

\begin{enumerate}
    \item Hamilton provided a straightforward generalization of Bring's quintic solution to the sextic case using Tschirnhaus transformations \cite{hamilton1836inquiry}. This formula's normal form is the surface $V_{\Ham}$.
    \item Klein outlined a generalization of his method for solving quintics to the sextic case by replacing the icosahedral $A_5$-action on $\PP^1$ with the exceptional Valentiner $A_6$-action on $\PP^2$ \cite{klein1905auflosung,sutherland2019felix}. This formula's normal form is the surface $V_{\Kle}$.
    \item Quite recently, Farb, Kisin, and Wolfson gave an alternate solution to the general sextic using Tschirnhaus transformations that differ from Hamilton \cite{farb2023modular}; we call this the \emph{modular solution}. This formula's normal form is the surface $V_{\Mod}$.
\end{enumerate}
These surfaces are used to build formulas for the roots of sextics out of algebraic functions of at most $2$ variables. We give a modern overview of each formula in \Cref{Section2}. In light of Green's work on the quintic, it is natural to ask how these normal forms, and their associated formulas for the roots of sextics, compare and relate to each other. We reformulate \Cref{MainTheoremOnMaps} in the context of solutions to the general sextic, and give its restatement below:

\begin{theorem}[Incompatibility of sextic solutions]\label{MainTheorem}
    There are no $A_6$-equivariant generically finite dominant rational maps between any of the normal forms defining solutions to the general sextic polynomial. Thus there is no algebraic transformation from one formula for the roots of a sextic into any other formula.
\end{theorem}

We now describe the content of each section of this paper. In \Cref{Section2}, we give an overview of the theory of Tschirnhaus transformations, and formally define algebraic formulas and their normal forms. We also describe each solution to the general sextic, and the tower of operations required to define their respective formulas. In \Cref{Section3}, we discuss the geometry of each formula's normal form, and then we obstruct the existence of equivariant maps between them in \Cref{Section4}.

In \Cref{Section5}, we give a coda to Green's uniformization of the Bring and Klein normal forms by explaining how one can use the uniformization of $V_{\Mod}$ by $\HH^2 \times \HH^2$ to analytically solve sextic polynomials using Hilbert modular forms. This expands on prior work of Farb--Kisin--Wolfson \cite[Proposition 4.14 (3)]{farb2023modular}, and offers a formula for the roots of a general sextic in terms of the uniformizing differential equation of the Hilbert modular surface $V_{\Mod}$.

\subsection{Acknowledgements} 
I would like to acknowledge the contributions of Alex Sutherland to this work. Alex introduced me to this problem early in graduate school, and although his efforts shaped the foundations of this work, he has declined to be a coauthor. I thank him for all he has taught me about resolvent problems. I thank Claudio G\'omez-Gonz\'ales and Jesse Wolfson for comments and corrections on earlier drafts, and for insightful conversations over the years. I thank Patrick Brosnan, Steven Creech, Trent Lucas, and Daniel Minahan for helpful discussions. I acknowledge Akiyoshi Sannai for pointing out small mathematical typos in a previous version. Finally, I thank Thomas Brazelton and Benson Farb for their careful review and numerous suggestions which made this article more readable. This work was supported in part by NSF Grants No. DMS-1944862 and DMS-2503485.

\subsection{Statement of AI usage}
The first version of the paper was written without the use of any AI tools. Claude Opus was used to check for mathematical correctness and grammar in the second version. I claim full responsibility for the content of this article.

\section{Solving generic sextics}\label{Section2}

In this section, we will describe how the three algebraic solutions to the general sextic equation work, and write down the tower that describes each formula. Before describing the three known methods to solve for the roots of a general sextic, we will introduce some fundamental ideas, such as the definition of a formula, and various facts regarding Tschirnhaus transformations, as they are used in all three approaches.

\begin{definition}\label{branchedCoverDefinition}
    A rational map of irreducible varieties $f: Y\to X$ is \emph{dominant} if the image of $f$ is Zariski dense in $X$; $f$ is \emph{generically finite} if the generic fiber is a finite set. A \emph{branched cover} of varieties $f: Y\to X$ is a generically finite dominant rational map.
\end{definition}

\begin{definition}[Formula for a cover]\label{formulaDef}
Let $Y\to X$ be a branched cover of $K$-varieties. A \emph{formula} for $Y \to X$ in at most $d$ variables is a tower of branched covers
$$E_k \to E_{k-1} \to \dots \to E_1 \to E_0 \subseteq X$$
along with primitive elements of the function field extensions $\alpha_i\in K(E_{i})/K(E_{i-1})$ such that 
\begin{enumerate}
    \item $E_0 \subseteq X$ is a dense Zariski open,
    \item the total composition $E_k \to E_0$ factors through $Y\to X$, and
    \item for each morphism $\pi_i : E_{i+1} \to E_i$, there exists a branched cover $\tilde{Z}_i \to Z_i$ with $\dim(Z_i) \leq d$, a Zariski open $U_i \subseteq E_i$, and a morphism $E_i \to Z_i$ such that $E_{i+1}|_{U_i} \cong U_i \times_{Z_i} \tilde{Z}_i$.
\end{enumerate}
\end{definition}

\begin{definition}[Normal form]\label{NormalFormDefinition}
    Given a formula for a branched cover of varieties $Y\to X$ as in \Cref{formulaDef}, the \emph{normal form} associated to this formula is the variety $\tilde{Z}_{k-1}$.\footnote{One can define a normal form of a cover to be an $\mathcal{E}$-versal variety of minimal known dimension, where $\mathcal{E}$ is a chosen class of accessory irrationalities --- in this article, we let $\mathcal{E}$ be the class of solvable covers. We do not use the language of versality in the remainder of this paper, and refer to \cite[Definition 3.8]{farb2020resolvent} and \cite[Definitions 4.1 and 4.4]{farb2023modular} for more details.}
\end{definition}

\begin{example}[Cardano's formula]\label{CardanoExample}
    Let $\Poly_3'$ be the space of reduced cubics $\alpha(z)=z^3 + p z + q$, and let $\widetilde{\Poly}_3'$ be the space of tuples $\{(\alpha(z),\lambda_1,\lambda_2,\lambda_3 )\in \Poly_3' \times \CC^3: \alpha(\lambda_1) = \alpha(\lambda_2) =\alpha(\lambda_3) = 0\}$. We realize Cardano's formula for the roots of a cubic as a tower of pullback covers:
    
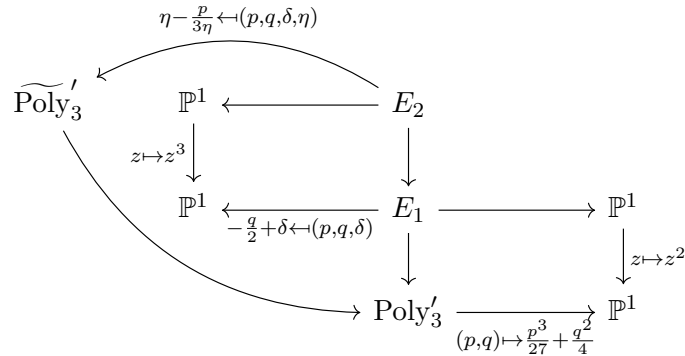
\begin{center}
    \begin{tikzcd}
\widetilde{\Poly}_3' \arrow[rrrdd, bend right] & \PP^1 \arrow[d, "z\mapsto z^3"'] &  & E_2 \arrow[d] \arrow[lll, "{\eta-\frac{p}{3\eta} \mapsfrom (p,q,\delta,\eta)}"', bend right] \arrow[ll] &  &                                 \\
                                               & \PP^1                            &  & E_1 \arrow[d] \arrow[ll, "{-\frac{q}{2} + \delta\mapsfrom (p,q,\delta)}"] \arrow[rr]                    &  & \PP^1 \arrow[d, "z\mapsto z^2"] \\
                                               &                                  &  & \Poly_3' \arrow[rr, "{(p,q)\mapsto \frac{p^3}{27} +\frac{q^2}{4}}"']                                    &  & \PP^1                          
\end{tikzcd}
\end{center}
Following the conventions established in \Cref{NormalFormDefinition}, the normal form of Cardano's formula is $\tilde{Z}_1 = \PP^1$, which comes with the $\ZZ/3\ZZ$-action generating a cube root (branched) cover of $Z_1 = \PP^1$. 
\end{example}

\begin{remark}[On labeling maps and primitive elements]
    In \Cref{CardanoExample}, we have slightly abused notation when labeling the map $E_2 \to \Poly_3'$ to have image the root $\lambda = \eta - \frac{p}{3\eta}$ --- the other roots of $\alpha(z) = z^3 +pz + q$ are implicit, and are obtained by multiplying the cube root $\eta$ by a third root of unity.\footnote{Since the cubic $\alpha(z)$ is reduced, any two roots determine the third by the relation $\lambda_3 = -\lambda_1 - \lambda_2$.} As writing down full towers of pullback squares is quite notationally dense, we will often suppress the explicit cover we pull back along and instead label the map $E_i \to E_{i-1}$, e.g. the label $\sqrt{\Delta}$ will denote the square root of the discriminant.

    Let us also explain the purpose of picking primitive elements $\alpha_i$ for each cover $E_i \to E_{i-1}$ which appear in a formula for $Y \to X$. By construction, the branched cover $E_k \to Y$ will be a rational function $R \in K(X)(\alpha_1,\dots,\alpha_k)$. The primitive elements allow us to explicitly name such a map. In \Cref{CardanoExample}, the primitive elements are a distinguished square root of the discriminant, denoted by $\delta$, and a distinguished cube root of $-\frac{q}{2} +\delta$, denoted by $\eta$.
\end{remark}

Let $K_n = \CC(c_1,\dots,c_n)$ denote the purely transcendental extension of $\CC$ with transcendence basis $c_1,\dots,c_n$. A \emph{general degree $n$ polynomial} is the element of the polynomial ring 
$$p_n(z) = z^n + c_1 z^{n-1} + \dots+ c_{n-1}z + c_n \in K_n[z].$$
We will denote the roots of the general polynomial $p_n(z)$ by the indeterminants $z_1,\dots,z_n$.

Recall that the $S_n$-cover $\Spec(\CC(z_1,\dots,z_n)) \to \Spec(\CC(c_1,\dots,c_n))$ maps the $n$ indeterminants $z_1,\dots,z_n$ to the polynomial whose roots are $z_1,\dots, z_n$; that is, $c_k = \sigma_k(z_1,\dots,z_n)$, where $\sigma_k$ is the $k$th elementary symmetric polynomial. This is the central map we are interested in constructing formulas for. We now describe one of the main techniques used to construct formulas for this class of covers. 

\begin{definition}
    A \emph{Tschirnhaus transformation} $\Upsilon$ is an isomorphism of $K_n$-fields
    $$\Upsilon: K_n[z]/(p_n(z)) \cong K_n[z]/(q_n(z)),$$
    where $q_n(z) = z^n + b_1z^{n-1} + \dots + b_{n-1}z + b_n$ and the coefficients $b_i$ are polynomials in the coefficients $c_i$. A Tschirnhaus transformation $\Upsilon$ has \emph{type} $(i_1,\dots,i_k)$ if $b_{i_1} =\dots = b_{i_k} = 0$.
\end{definition}

As explained by Sutherland \cite{sutherland2021upper} and Wolfson \cite{wolfson2021tschirnhaus}, the primitive element theorem implies that the (extended) space of Tschirnhaus transformations is $\A_{K_n}^n$.\footnote{The adjective ``extended" here simply refers to the fact that there are some trivial Tschirnhaus transformations we wish to ignore along the first coordinate's affine line.} Since we only need to consider Tschirnhaus transformations up to re-scaling, we instead consider the projectivized Tschirnhaus transformation space $\mathcal{T}_{K_n} = \PP_{K_n}^{n-1}$. 

\begin{definition}
    A type $(i_1,\dots,i_k)$ \emph{Tschirnhaus complete intersection} $\tau_{i_1,\dots,i_k} \subseteq \mathcal{T}_{K_n}$ is the complete intersection $\tau_{i_1,\dots,i_k} = Z(b_{i_1}, b_{i_2},\dots, b_{i_k})$. 
\end{definition}

Let us explain how Tschirnhaus complete intersections are useful for extracting roots from polynomials. Suppose we are given a type $(1,\dots,d)$ Tschirnhaus transformation $\Upsilon$ of $p_n(z)$. Then we need only find the roots of the algebraic function
$$q_n^d(z) = z^n +b_{d+1}z^{n-d-1} +\dots + b_{n-1}z+b_n,$$
as $\Upsilon^{-1}$ takes the roots of $q_n^d(z)$ to the roots of $p_n(z)$. Furthermore, by considering the transformation up to projective equivalence, we cut down one independent variable when solving for the roots of $q_n^d(z)$. In other words, $\Upsilon$ determines a formula for the roots of $p_n(z)$ in $n-d - 1$ variables. Finally, to ensure non-triviality of the Tschirnhaus transformations $\Upsilon$ used, we will want to avoid the point $[1:0:\dots:0] \in \tau_{i_1,\dots,i_k}$. A contemporary treatment of Tschirnhaus transformations is given in \cite[Section 3]{wolfson2021tschirnhaus}.

When building a formula for a branched cover $X\to Y$, not every step in the tower $E_i \to E_{i-1}$ reduces the Galois group of $X\to Y$. Such transformations are called \emph{accessory irrationalities}.\footnote{The notion of an accessory irrationality appeared across Klein's corpus, see for example \cite{klein1884vorlesungen,klein1958lectures}. As far as I know, the first formal definition was given in \cite[Definition 4.1]{farb2023modular}.} When solving for the roots of a general polynomial $p_n(z)$, one example of an accessory irrationality is a type $(1,2)$ Tschirnhaus transformation, which solves an auxiliary quadratic to cancel the coefficients $c_1 = c_2 = 0$. A non-example is the square root of the discriminant $\sqrt{\Delta}$, which cuts the Galois group down from $S_n$ to $A_n$; this step will appear at the beginning of every formula described.

\subsection{Hamilton's approach}

Given a general monic polynomial of degree six 
$$p(z) = z^6 + c_1 z^5 + c_2z^4 + c_3z^3 + c_4z^2+ c_5 z + c_6,$$
our goal is to determine a root of $p(z)$ in the simplest possible manner. Simplest here should be interpreted as meaning ``using the least number of independent variables", in the spirit of the resolvent problem of reducing to the minimal number of algebraic parameters needed. As such, we are going to solvably reduce this problem to solving for the roots of 
$$q(z) = z^6 + b_4z^2+ b_5 z + b_6,$$
where $b_4$, $b_5$, and $b_6$ have yet to be determined. Thus, if we adjoin to our list of operations the algebraic function that produces the roots of $q(z)$, i.e. determine a point on the Tschirnhaus complete intersection $\tau_{1,2,3}$, we can solve for the roots of all sextic polynomials $p(z)$. We follow the exposition of Green closely \cite[Section 2]{green1978analytic}.

Suppose that $z_1 , \dots , z_6$ are the roots of our original polynomial $p(z)$. To find the values of $z_1,\dots,z_6$, it suffices to solve for the roots of the polynomial
$$q(z) = \prod_{k=1}^6 (z - y_k)$$
where we define the roots $y_k$ of $q(z)$ by
$$y_k = \sum_{n=0}^4 a_n z_k^n.$$
Here we have yet to define the coefficients $a_n$. To see why this change of coordinates is useful, note that the $a_n$'s determine an invertible matrix that takes the $z_k$'s to the $y_k$'s. Thus the problem of finding the roots of $q(z)$ is solvably equivalent to finding the roots of $p(z)$.

We want to impose the following condition on the alternative polynomial $q(z)$: the coefficients on $z^5$, $z^4$, and $z^3$ must all vanish. This is equivalent to requiring that the $1$st, $2$nd, and $3$rd elementary symmetric polynomials in $y_k$ must vanish. Under a change of coordinates given by the Newton identities, we thus have that
\begin{align*}
    0 &= \sum_{k=1}^6 y_k= \sum_{k=1}^6 y_k^2 = \sum_{k=1}^6 y_k^3.
\end{align*}
Set $s_n = \sum z_k^n$. Substituting the definition $y_k = \sum a_n z_k^n$ into the above equations, we obtain three equations, called $H$ (hyperplane), $Q$ (quadric), and $C$ (cubic), which are all homogeneous in the quantities $a_n$, with coefficients given by the $s_n$. 

To form this reduction then is to find some point $[a_0: a_1 : \dots : a_4] \in \PP^4$ lying on the complete intersection $H \cap Q \cap C$. We pass to the hyperplane $H$ to instead work in $\PP^3$, i.e. we are finding a point on the intersection of $Q' = H\cap Q$ and $C' = H \cap C$. We describe the geometry of this process below, and note that this general method is described explicitly in \cite[Section 1]{wolfson2021tschirnhaus}. 

As mentioned before, the trivial solution $[1:0:0:0:0] \in\PP^4$ is not a valid point for our construction, as this corresponds to defining a value for $y_k$ independent of the original roots $z_k$. We wish to find a point $[a_0: \dots: a_4] \in \PP^4$ on the complete intersection not equal to the trivial solution. 

In particular, we want to solvably determine this point, which requires us using polynomials of degree no more than 4. This is a problem, since this projective variety is degree 6. Thus we will introduce an accessory quadratic irrationality, in the form of a tangent hyperplane $H'$.

We can find a hyperplane $H'$ not containing $[1:0:0:0:0] \in\PP^4$ that is tangent to $Q'$ by solving a quadratic equation. Since $H'$ is tangent to $Q'$, we have that $H' \cap Q'$ is a singular conic, and it comes equipped with one double point (this is by Bezout's theorem). Thus it is the union of two lines $L_1$ and $L_2$. To determine the difference between $L_1$ and $L_2$, we need to solve another quadratic equation.

It follows that $H' \cap Q' \cap C' = (L_1 \cap C) \cup (L_2 \cap C)$. To determine a point on $L_1 \cap C$, we need only solve a cubic equation. This introduces another square root and cube root to our adjoined irrationalities. This renders our desired point $[a_0:\dots: a_4] \in \PP^4$ which is not the trivial solution $[1: 0: 0:0:0]\in \PP^4$.

This procedure found us a degree $2^3 \cdot 3 = 24$ formula which reduces the problem of finding the roots of a generic sextic $p(z)$ to finding the roots of a reduced sextic $q(z)$. When we adjoin the algebraic function that solves a general $q(z)$, i.e. the function which finds a point on $\tau_{1,2,3}$, we have produced a solution to the generic sextic.

Before we write this solution as a tower, let us establish some notation. Along the maps $E_i \to E_{i-1}$, we will label the map by the corresponding deck group, or by the operation we adjoined (e.g. add in a square root $\sqrt{\Delta}$, solving a quadratic $\phi_2$, or more generally solving a degree $n$ polynomial $\phi_n$). The branched cover $\tilde{Z}_i\to Z_i$ we pull back along is implicit in the tower below (see Definition \ref{formulaDef}).

With this notation established, the tower of covers defining Hamilton's solution to the sextic can now be written down:
$$ E_5 \xrightarrow[]{A_6} E_4\xrightarrow[]{\phi_3}E_3\xrightarrow[]{\sqrt{-}} E_2 \xrightarrow[]{\sqrt{-}}E_1 \xrightarrow[]{\sqrt{\Delta}} \Spec(K_6).$$
\subsection{The modular approach}

This is almost identical to Hamilton's approach, and was first identified as an alternate solution to the general sextic in \cite[pp. 137-138]{farb2023modular}. Instead of setting $b_1 = b_2 = b_3 = 0$, we reduce generic sextics down to a form with $b_1 = b_2  = b_4 = 0$. In other words, our Tschirnhaus transformation into the $y_k$'s satisfies
\begin{align*}
    0 &= \sum_{k=1}^6 y_k = \sum_{k=1}^6 y_k^2 = \sum_{k=1}^6 y_k^4.
\end{align*}
The last step of Hamilton's solution,  solving a cubic, is then replaced with solving a quartic. To solve a degree 4 equation $\phi_4$, we need three square roots and a cube root (this can be seen by looking at the $S_4$-cover corresponding to Ferrari's formula). Adjoining the algebraic function that finds a point on the Tschirnhaus complete intersection $\tau_{1,2,4}$ completes the modular formula for the roots of the general sextic. As a tower, we have the formula 
$$ F_5 \xrightarrow[]{A_6} F_4\xrightarrow[]{\phi_4} F_3 \xrightarrow[]{\sqrt{-}}F_2 \xrightarrow[]{\sqrt{-}}F_1 \xrightarrow[]{\sqrt{\Delta}} \Spec(K_6).$$

\subsection{The quintic and the icosahedron} Before we describe Klein's approach to solving sextics, it will be helpful to first sketch his solution to quintics, which the icosahedron is at the center of. Adjoining the square root of the discriminant $\sqrt{\Delta}$ of a general quintic reduces the Galois group from the symmetric group $S_5$ to the alternating group $A_5$. The icosahedron, which has isometry group isomorphic to $A_5$, can be inscribed in the projective line $\PP^1$. This induces the icosahedral $A_5$-action on $\PP^1$, and defines a branched cover $\PP^1 \to \PP^1/A_5$. Klein calculated the ring of invariants of this $A_5$-action on $\PP^1$ to prove the isomorphism of varieties $\PP^1 / A_5 \cong \PP^1$; this quotient is given by a degree $60$ rational map  $\mathcal{I}([z_0:z_1])$ called the \emph{icosahedral function}. 

The key idea of Klein is that the general quintic can be reduced to a normal form for which inverting the icosahedral function $\mathcal{I}$ allows one to calculate all the roots of the quintic. To execute this proposed formula, one needs to extract an accessory square root, and to determine the correct invariants on $\PP^1$ associated to the quintic needed in order to solve for the roots of the quintic; the latter was accomplished by Klein when he calculated the ring of invariants. O. Nash gives a nice modern exposition of Klein's solution to the quintic, including how to produce the \emph{icosahedral invariant} $Z$ attached to a given quintic \cite{nash2014klein}. Thus to solve for the roots of a quintic with icosahedral invariant $Z$, one requires an accessory square root, the square root of the discriminant $\sqrt{\Delta}$, and the solution $z$ to the equation
$$\mathcal{I}(z) = Z.$$
In a tower, Klein's formula is given by
$$G_3 \xrightarrow[]{\mathcal{I}} G_2 \xrightarrow[]{\sqrt{-}} G_1 \xrightarrow[]{\sqrt{\Delta}} \Spec(K_5).$$

To invert the icosahedral function, an algebraic function, requires transcendental function theory (i.e. elliptic modular functions); this analytic work was already carried out by Hermite and other contemporary mathematicians like Brioschi and Kronecker. In fact, this classical work amounts to saying that the icosahedral cover is modular, in the sense that the cover $\PP^1\to \PP^1/A_5$ is $A_5$-equivariantly birational to the congruence cover $\HH^2/\SL_2(\ZZ)[5] \to \HH^2/\SL_2(\ZZ)$, where $\SL_2(\ZZ)[5]$ is the level $5$ congruence subgroup of the modular group.\footnote{Note we are using the exceptional isomorphism of groups $\PSL_2(\FF_5) \cong A_5$.} Klein viewed this as interesting secondary phenomena; to him, the algebraic part of solving the quintic was the core issue \cite{klein1905auflosung}. We will now describe the analogous algebraic part in Klein's solution to the general sextic polynomial.

\subsection{Klein's approach}\label{KleinsApproach}

Given the general sextic polynomial
$$z^6 + c_1 z^5 + c_2z^4 +c_3z^3 + c_4z^2 + c_5 z + c_6 = 0,$$
the first step of Klein's solution to the sextic, much like his formula for the roots of a quintic equation, is to reduce the sextic to a \emph{principal normal form}:
$$p(z) = z^6 +a_3z^3 + a_4z^2 + a_5 z + a_6 = 0.$$
This reduction is done via a degree $2$ Tschirnhaus transformation. We will also extract the square root of the discriminant $\sqrt{\Delta}$ to reduce the Galois group from $S_6$ to $A_6$. From this point onwards, there are no more similarities to the Tschirnhaus approach of root extraction --- Klein's approach to solve for the roots of a sextic is analogous to his solution for the roots of a quintic via the icosahedral function.

To solve sextics with the Kleinian approach, we need an $A_6$-invariant algebraic function analogous to the icosahedral function used to solve quintics. To do this, we replace the $A_5$-action on $\PP^1$ which appeared in the quintic case with an $A_6$-action on $\PP^2$ called the \emph{Valentiner action}. We will describe this $A_6$-action on $\PP^2$ now. 

Recall that the Schur multipliers of the alternating groups are
$$H^2(A_n;\CC^*) \cong \begin{cases}
    \ZZ/2\ZZ & \text{ if } n\geq 4 \text{ and } n\neq 6,7,\\
    \ZZ/6\ZZ & \text{ if } n=6,7.
\end{cases}$$
It is no accident that $A_6$ has an order $6$ Schur multiplier; there is a degree $3$ extension of $A_6$, denoted by $\mathcal{V} = 3.A_6$, which admits a faithful special linear action on $\CC^3$. This representation was first discovered by Valentiner \cite{valentiner1889endelige}, and soon after Wiman showed that this extension descends to a faithful representation $A_6 \to \PGL_3(\CC)$ with generating set representatives
$$Z = \begin{pmatrix*}
-1 & 0 & 0\\
0 & 1 & 0\\
0 & 0 & -1
\end{pmatrix*}, P = \frac{1}{2}\begin{pmatrix*}
1 & \tau^{-1} & -\tau\\
\tau^{-1} & \tau & 1\\
\tau & -1 & \tau^{-1}
\end{pmatrix*}, Q = \begin{pmatrix*}
1 & 0 & 0\\
0 & 0 & \rho^2\\
0 & -\rho & 0
\end{pmatrix*},$$
where $\tau = (1+\sqrt{5})/2$ and $\rho = e^{2\pi i /3}$ \cite{wiman1896ueber}. The relationship between these generators and various icosahedral and tetrahedral symmetries within the Valentiner group is carefully traced out in Crass--Doyle \cite{crass1997solving}.

It was already known to Wiman, and described by Klein, that to pass from the Valentiner action $\mathcal{V} = 3.A_6$ on $\CC^3$ to the alternating group action $A_6$ on $\PP^2$, one must adjoin an additional cube root --- the root we adjoin is used to find an inflection point on a canonical cubic curve $C \subset \PP^2$ associated to $p(z)$. The construction of such a curve is described in \cite[Step 8]{sutherland2019felix},\footnote{Note that Klein's outlined construction reorders when we adjoin the accessory square root which we used to put a general sextic in its principal form.} and the accessory cube root in \cite[Section 6]{crass1997solving}. 

Now that we have adjoined all natural and accessory irrationalities necessary, we can describe the $A_6$-invariant \emph{Valentiner function} $Y_\mathcal{V}$ on $\PP^2$ which defines the quotient map $\PP^2 \to \PP^2/A_6$. The inversion of $Y_\mathcal{V}$ will allow us to solve sextic polynomials, just as inversion of the icosahedral function $\mathcal{I}$ allowed us to solve quintic polynomials.

The ring of $A_6$-invariants on $\PP^2$ was worked out explicitly by Wiman and Klein, and is generated in degrees $6$, $12$, and $30$: the invariant forms in question are labeled $F$, $\Phi$, and $\Psi$, respectively, and are explicitly written out in coordinates on \cite[p. 221]{Crass1999}. By reindexing the grading, Klein showed that $\PP^2 /A_6 \cong \PP^2$. Klein and Fricke explicitly wrote the rational map $Y_\mathcal{V}$ in homogeneous coordinates (\cite[p. 19]{sutherland2019felix} and \cite[p. 302]{fricke1924lehrbuch}); it is expressed in terms of the invariant forms like so:
$$Y_\mathcal{V}(z) := [F(z)^3 \Phi(z) : \Psi(z) : F(z)^5].$$

The \emph{Valentiner invariant} $[v_p:w_p:1] \in \PP^2$ associated to the given sextic polynomial $p(z)$ can be calculated in terms of the coefficients of $p(z)$, the square root of the discriminant, and the accessory cube root; the extraction of this invariant is described in \cite{crass1997solving}, with more details found in \cite{fricke1924lehrbuch}. 

To conclude, Klein's formula for the roots of a sextic involves adjoining two square roots and a cube root, giving us a degree $2^2 \cdot 3 = 12$ formula before adjoining the Valentiner function. This gives us the following tower to solve for the roots of a general sextic:
$$H_4 \xrightarrow[]{A_6} H_3 \xrightarrow[]{\sqrt[3]{-}} H_2 \xrightarrow[]{\sqrt{-}} H_1 \xrightarrow[]{\sqrt{\Delta}} \Spec(K_6).$$
\section{The geometry of normal forms}\label{Section3}

Having described the three known solutions for the roots of a general sextic polynomial, we now go over the geometric structure of each algebraic surface underlying the normal forms of each solution. Previously, all normal forms considered were thought of as $K_6$-varieties --- we now change base and study the geometry of each surface \emph{as a $\CC$-variety}.

\begin{proposition}
Hamilton's normal form $V_{\Ham}$, given by the Tschirnhaus complete intersection $\tau_{1,2,3}$, is a nonsingular K3 surface.
\end{proposition}
\begin{proof}
    Nonsingularity follows from the Jacobi criterion. Triviality of the canonical bundle follows from a simple adjunction calculation.
\end{proof}

\begin{proposition}[{\cite[Chapter VIII, Theorem 2.6]{van1988hilbert}}]\label{vanderGeerBirational}
The modular normal form $V_{\Mod}$, given by the Tschirnhaus complete intersection $\tau_{1,2,4}$, is birationally a Hilbert modular surface.
\end{proposition}
\begin{proof}
The birational equivalence is induced by the cusp forms, and will be discussed in \Cref{Section5}. The Hilbert modular surface in question is given by $\Gamma_5[3]\backslash \HH^2\times \HH^2$, where
$$\Gamma_{5}[3] = \ker \left(\SL_{2}\left(\ZZ\left[\frac{1+\sqrt{5}}{2}\right]\right)\rightarrow \PSL_2(\FF_9)\right).$$
Note that the choice of a level $3$ structure is no accident, as Fricke showed that $\PSL_2(\FF_9) \cong A_6$ \cite{fricke1924lehrbuch}.
\end{proof}

\begin{proposition}\label{HodgeDiamond}
The Tschirnhaus complete intersection $\tau_{1,2,4}$ has $30$ simple $A_1$-singularities and Hodge number $h^{2,0} = \dim(H^0(\tau_{1,2,4},\Omega^2)) = 5$, and degree $K_{\tau_{1,2,4}}^2 = 8$.
\end{proposition}

\begin{proof}
The singularity type and degree is worked out in detail by Hirzebruch and van der Geer \cite[5.5, Example II]{hirzebruch1981lectures}. The Hodge number computation follows from \cite[\'{E}xpos\'{e} XI, Theorem 2.3.]{deligne2006groupes} and birational invariance of $h^{2,0}$. 
\end{proof}

Although Klein's solution is certainly the most interesting of the three solutions, it does not have a particularly exotic normal form; we include the proposition below for completeness:

\begin{proposition}
    Klein's normal form $V_{\Kle}$ is the projective plane $\PP^2$. 
\end{proposition}

\section{Obstructing maps of formulas via geometry}\label{Section4}

We now come to the heart of this work, showing that the three simplest known formulas for the roots of sextics are fundamentally incompatible: concretely, we will prove Theorem \ref{MainTheorem}, showing that there are no equivariant branched covers between any of their normal forms. Before obstructing the existence of maps that relate the distinct solutions of the sextics, we begin with some generalities that will apply to our proofs. In what follows, we will often trade between the notations $V_{\Ham} = \tau_{1,2,3}$, $V_{\Mod} = \tau_{1,2,4}$, and $V_{\Kle} = \PP^2$, depending on whether we want to emphasize the surface's role as a normal form or its construction as a complete intersection in the appropriate projective space.

\begin{lemma}\label{FromKnToC}
    If there were an $A_6$-equivariant dominant rational map between any of the normal forms $V_{\Ham}$, $V_{\Mod}$, or $V_{\Kle}$ as $K_n$-varieties, there would be an induced $A_6$-equivariant map of $\CC$-varieties.
\end{lemma}
\begin{proof}
    This is a simple consequence of base change: given any field extension $K \subseteq L$ and a $K$-scheme $X$, let $X_L :=
X_K \times_{\Spec(K)}\Spec(L)$. Any morphism $\varphi_K : X_K \to Y_K$ of schemes over $K$ induces a unique morphism $f_L : X_L \to Y_L$ of schemes over $L$ by the universal property of fiber products. Given such a morphism of normal forms over $K_n$, it would induce a unique morphism over the algebraic closure $\overline{K_n} \cong \CC$. To briefly explain the last (non-canonical) isomorphism, recall that for a fixed characteristic, there is exactly one model up to isomorphism for an algebraically closed field of a given uncountable cardinality \cite{steinitz1910algebraische}. Thus any characteristic $0$ algebraically closed field of cardinality $2^{\aleph_0}$ is isomorphic to the complex numbers.
\end{proof}

To obstruct any connections between the three algebraic formulas for roots of sextics, we must prove that no $A_6$-equivariant maps between their normal forms exist. By Lemma \ref{FromKnToC}, it suffices to obstruct $A_6$-equivariant branched covers between these varieties over the complex numbers.

\subsection{Obstructing maps from the modular normal form}

Recall that our modular normal form is defined by the Tschirnhaus complete intersection $V_{\Mod} = \tau_{1,2,4} \subset \PP^5$. We would like to obstruct $A_6$-equivariant maps to Klein's and Hamilton's normal forms.\footnote{As mentioned in the introduction, maps in the other direction are not considered, as it would also be impossible for such varieties to map dominantly onto a surface of general type.} Our first obstruction will arise by viewing these spaces' associated cohomology groups as $A_6$-representations.

\begin{proposition}
The space of global sections of holomorphic 2-forms $H^0(\tau_{1,2,4},\Omega^2)$ is isomorphic to an irreducible 5-dimensional representation of $A_6$.
\end{proposition}
\begin{proof} 
\Cref{HodgeDiamond} tells us that $H^0(\tau_{1,2,4},\Omega^2)$ is $5$-dimensional. Moreover, it defines the canonical embedding $\tau_{1,2,4} \to \PP H^0(\tau_{1,2,4},\Omega^2) \cong \PP^4$, upon which $A_6$ acts nontrivially, as it comes from the permutation representation on $\PP^5$, which passes through taking the hyperplane section into $\PP^4$. As the $A_6$-action is nontrivial, $H^0(\tau_{1,2,4}, \Omega^2)$ is isomorphic to an irreducible $5$-dimensional representation of $A_6$, as claimed. By inspection, it is the one which comes from the permutation representation; for a classification of irreducible $A_6$-representations, see \cite[Part I]{serre1977linear}.
\end{proof}

We now argue using Schur's lemma that there exists no equivariant branched cover $V_{\Mod} \to V_{\Ham}$:

\begin{theorem}\label{NoModularHamilton}
There exists no $A_6$-equivariant generically finite dominant rational map $V_{\Mod}\to V_{\Ham}$.
\end{theorem}
\begin{proof}
Assume for the sake of contradiction there exists an $A_6$-equivariant map $h : \tau_{1,2,4} \to \tau_{1,2,3}$. The induced map on cohomology
$$h^* : H^0(\tau_{1,2,3},\Omega^2) \to H^0(\tau_{1,2,4},\Omega^2)$$
must also be $A_6$-equivariant. Moreover, we claim that $h^*$ is injective. Suppose for the sake of contradiction that $h^*$ was not injective, and let $\omega$ be a representative generator of $H^0(\tau_{1,2,3},\Omega^2)$. Taking conjugates and cup product yields a volume form $\omega \wedge \Bar{\omega}$ on the K3 surface $\tau_{1,2,3}$. By naturality of cup products, $h^*(\omega \wedge \Bar{\omega}) = h^*\omega \wedge h^* \Bar{\omega} = 0$, i.e. the map between volume forms is the zero map. This would be a contradiction, as $h$ is dominant, so $h^*$ is indeed injective.

Given two irreducible representations $V$ and $W$, Schur's lemma tells us that $\Hom_{A_6}(V,W)$ is $1$-dimensional when $V \cong W$ and, $0$-dimensional otherwise. Since these $A_6$-representations aren't isomorphic (because $\dim(H^0(\tau_{1,2,3},\Omega^2)) = 1 \neq 5 = H^0(\tau_{1,2,4},\Omega^2)$), the space of homomorphisms between them is $0$-dimensional. Thus the induced map on $H^0(\Omega^2)$ had to be the zero map, so no dominant rational map between the normal forms could have existed.
\end{proof}

To obstruct an equivariant branched cover $V_{\Mod}\to V_{\Kle}$, we will use a more sophisticated tool coming from equivariant birational geometry, namely an invariant called the \emph{Amitsur group}. Amitsur groups arise when studying linearizations of line bundles and are a useful equivariant
birational invariant \cite{blanc2023finite}. As far as we are aware, our use of the Amitsur group to obstruct equivariant dominant rational maps outside of the realm of equivariant unirationality is novel. For further background on invariants in equivariant birational geometry and their applications, see \cite{kresch2025unramified} and \cite{kresch2026invariants} and citations therein.

Given a smooth projective variety $X$ with a finite $G$-action, a line bundle $\mathcal{L}$ on $X$ is \emph{linearizable} if there exists a linear action of
$G$ on the total space of $\mathcal{L}$ that is compatible with the $G$-action on $X$.
A choice of such action is a \emph{linearization} of $\mathcal{L}$ and the set of isomorphism
classes of line bundles together with a choice of linearization forms a group $\Pic(X, G)$.

The Leray spectral sequence for Deligne--Mumford stacks yields an exact sequence
\begin{equation}\label{eq:LerayExactSequence}
    0 \to \Hom(G,\CC^*) \to \Pic(X,G) \to \Pic(X)^G \xrightarrow[]{\delta_2} H^2(G,\CC^*) \to \mathrm{Br}([X/G]),
\end{equation}
where $\Pic(X,G) = \Pic([X/G])$ denotes the subgroup of $G$-linearizable line bundles on $X$, $\Pic(X)^G$ denotes the $G$-invariant line bundles on $X$, and $\mathrm{Br}([X/G])$ denotes the Brauer group of the quotient stack $[X/G]$ (for further details, see \cite{kresch2025unramified}).
\begin{definition}
    The image of the boundary map $\delta_2$ is called the \emph{Amitsur group} of the pair $(X,G)$, which we denote by $\mathrm{Am}(X,G)$.
\end{definition}
Since the exact sequence \ref{eq:LerayExactSequence} comes from the Leray spectral sequence, it follows that it is contravariantly functorial in both $X$ and $G$. We record some basic properties of the Amitsur group:

\begin{lemma}\label{lem:AmitsurProperties}
    Fix $X$ and $Y$ to be smooth projective $G$-varieties.
    \begin{enumerate}
        \item If the $G$-action on $X$ has a fixed point, then $\mathrm{Am}(X,G) = 0$.
        \item If $X$ and $Y$ are $G$-equivariantly birational, then $\mathrm{Am}(X,G) = \mathrm{Am}(Y,G)$.
        \item Given a $G$-equivariant morphism $X\to Y$, suppose that the induced map $\Pic(Y)\to \Pic(X)$ is injective. Then $\Pic(Y,G)\to \Pic(X,G)$ is injective, and $\mathrm{Am}(Y,G)$ is contained in $\mathrm{Am}(X,G)$.
    \end{enumerate}
\end{lemma}
\begin{proof}
    Suppose that $X$ has a $G$-fixed point. By basic functoriality of the Leray spectral sequence, the map from $H^2(G,\CC^*)\to \mathrm{Br}([X/G])$ is injective, thus $\delta_2$ is the zero map and hence $\mathrm{Am}(X,G) = 0$. The second and third parts are the content of \cite[Theorem 6.1]{blanc2023finite} and \cite[Lemma 2.1]{kresch2025unramified}.
\end{proof}

It follows from the definitions that the canonical bundle of a $G$-variety is always $G$-linearizable. The following is an immediate consequence of the fact that the Valentiner $A_6$-action on $\PP^2$ is not the projectivization of a linear action: 

\begin{proposition}\label{prop:Klein-Amitsur-Group}
    The normal form $V_{\Kle}$ has Amitsur group $\mathrm{Am}(V_{\Kle},A_6) = \ZZ/3\ZZ$.
\end{proposition}

Given a smooth projective $G$-variety $X$, the behavior of Amitsur groups under restriction to subgroups $H \leq G$ can be quite subtle. The following proposition gives us a condition to check for injectivity on the $p$-primary part: 
\begin{proposition}\label{prop:injectivep-primaryAmitsur}
    Given a smooth projective $G$-variety $X$ and a subgroup $H\leq G$, suppose that the restriction map $H^2(G,\CC^*)_{(p)} \to H^2(H,\CC^*)_{(p)}$ on the $p$-primary part is injective. Then the associated map of Amitsur groups $\mathrm{Am}(X,G) \to \mathrm{Am}(X,H)$ is injective on the $p$-primary part.
\end{proposition}
\begin{proof}
    The result follows from looking at the commutative diagram
    \begin{center}
        \begin{tikzcd}
{H^2(G,\CC^*)_{(p)}} \arrow[r, hook]               & {H^2(H,\CC^*)_{(p)}}                     \\
{\mathrm{Am}(X,G)_{(p)}} \arrow[r] \arrow[u, hook] & {\mathrm{Am}(X,H)_{(p)}} \arrow[u, hook]
\end{tikzcd}
    \end{center}
    Injectivity of the vertical maps follows from the definition of $\mathrm{Am}(X,G)$, and so the bottom horizontal map is injective. 
\end{proof}

\begin{corollary}\label{cor:AmTrivial}
    Let $\widetilde{\tau}_{1,2,4}$ be the minimal resolution of the complete intersection $\tau_{1,2,4}$. The $3$-primary Amitsur group $\mathrm{Am}(\widetilde{\tau}_{1,2,4},A_6)_{(3)}$ is trivial.
\end{corollary}
\begin{proof}
    Let $P \cong (\ZZ/3\ZZ)^2\leq A_6$ denote the Sylow $3$-subgroup of $A_6$. By \cite[Chapter I, Proposition 6.9]{beyl1982group} (see also \cite[Satz V.25.1]{huppert2013endliche}), the restriction 
    $$H^2(A_6,\CC^*)_{(3)} \to H^2(P,\CC^*)_{(3)} \cong \ZZ/3\ZZ$$
    is injective (and in fact an isomorphism). It follows from \Cref{prop:injectivep-primaryAmitsur} that $\mathrm{Am}(\widetilde{\tau}_{1,2,4},A_6)_{(3)} \to \mathrm{Am}(\widetilde{\tau}_{1,2,4},P)_{(3)}$ is injective. To prove the claim, it suffices to check that the Amitsur group $\mathrm{Am}(\widetilde{\tau}_{1,2,4},P)_{(3)} = \mathrm{Am}(\widetilde{\tau}_{1,2,4},P)$ is trivial.

    To prove that $\mathrm{Am}(\widetilde{\tau}_{1,2,4},P) = 0$, we will use \Cref{lem:AmitsurProperties} (1) and show directly that the $P$-action on $\tau_{1,2,4}$ has a smooth fixed point, and hence the $P$-action on $\widetilde{\tau}_{1,2,4}$ has a fixed point. The explicit action of $P$ on $\tau_{1,2,4}$ is given by the group $P = \langle (1 \ 2 \ 3), ( 4 \ 5 \ 6)\rangle$ permuting coordinates in $\PP^5$. A direct calculation shows that the point $z_0 = [1: e^{2\pi i/3}: e^{4\pi i/3} : 0: 0: 0]$ lies on $\tau_{1,2,4}\subset \PP^5$ and is fixed by $P$. Moreover, the point $z_0$ lies outside the singular locus of $\tau_{1,2,4}$ \cite[5.5, Example II]{hirzebruch1981lectures}, and thus lifts to be a smooth global fixed point of the $P$-action on $\widetilde{\tau}_{1,2,4}$. This proves the claim. 
\end{proof}

\begin{theorem}\label{NoModularKlein}
    There exists no $A_6$-equivariant generically finite dominant rational map $V_{\Mod}\to V_{\Kle}$.
\end{theorem}
\begin{proof}
    Suppose for the sake of contradiction that such a map existed. By equivariantly blowing up the locus of indeterminacy and resolving the singularities on $V_{\Mod}$, we get an $A_6$-equivariant dominant morphism $f: \widetilde{V}_{\Mod} \to V_{\Kle}$. Consider the following commutative diagram induced by the exact sequence \ref{eq:LerayExactSequence}:
    \begin{center}
        \begin{tikzcd}
0 \arrow[r] & {\Pic(\widetilde{V}_{\Mod},A_6)} \arrow[r]       & \Pic(\widetilde{V}_{\Mod})^{A_6} \arrow[r]       & {\mathrm{Am}(\widetilde{V}_{\Mod},A_6)} \arrow[r]      & 0 \\
0 \arrow[r] & {\Pic(V_{\Kle},A_6)} \arrow[r] \arrow[u, "f^*"] & \Pic(V_{\Kle})^{A_6} \arrow[r] \arrow[u, "f^*"] & {\mathrm{Am}(V_{\Kle},A_6)} \arrow[u, "f^*"] \arrow[r] & 0
\end{tikzcd}
    \end{center}
    Note that the bottom row is isomorphic to the short exact sequence $0 \to \ZZ \to \ZZ \to \ZZ/3\ZZ \to 0$; this follows since the anticanonical class $-K_{\PP^2} = 3H$ spans $\Pic(V_{\Kle},A_6)$ and $H$ spans $\Pic(V_{\Kle})^{A_6}$ (see \Cref{prop:Klein-Amitsur-Group}). We will use the fact that $\mathrm{Am}(V_{\Kle},A_6) \cong \ZZ/3\ZZ$ to derive a contradiction.
    
    We first observe that the map $f^* : \Pic(V_{\Kle}) \to \Pic(\widetilde{V}_{\Mod})$ is injective, for instance by the projection formula. It follows from \Cref{lem:AmitsurProperties} (3) that the map $f^* : \mathrm{Am}(V_{\Kle},A_6) \to \mathrm{Am}(\widetilde{V}_{\Mod},A_6)$ is injective, and hence injective on the $3$-primary part. By \Cref{lem:AmitsurProperties} (2), $\mathrm{Am}(\widetilde{V}_{\Mod},A_6) = \mathrm{Am}(\widetilde{\tau}_{1,2,4},A_6)$, so by \Cref{cor:AmTrivial}, $\mathrm{Am}(\widetilde{V}_{\Mod},A_6)_{(3)} = \mathrm{Am}(\widetilde{\tau}_{1,2,4},A_6)_{(3)} = 0$. Thus the map on $3$-primary parts is $f^* : \ZZ/3\ZZ \to 0$, which contradicts the injectivity of $f^*$. Thus no such equivariant branched cover $V_{\Mod} \to V_{\Kle}$ existed. 
\end{proof}

\subsection{Obstructing maps from the Hamilton surface}

Recall that Hamilton's normal form is a smooth K3 surface $V_{\Ham} = \tau_{1,2,3} \subseteq \PP^5$. The algebraic surface $\tau_{1,2,3}$ is rather special, as it is equipped with a maximal, effective symplectic $A_6$-action --- this is simply the restriction of the linear $S_6$-permutation action on $\PP^5$. We say that a group $G$ acts \emph{symplectically} on a K3 surface $Y$ if the complex line $H^{2,0}(Y)$ in cohomology is a trivial $G$-representation. The maximality of this symplectic $A_6$-action on $\tau_{1,2,3}$ follows from Mukai's classification of finite symplectic automorphism groups of K3 surfaces, via an examination of their Mathieu representations \cite{mukai1988finite}.

We now use this maximal symplectic $A_6$-action on $\tau_{1,2,3}$ to compute the group $\mathrm{Am}(\tau_{1,2,3},A_6)$, which will in turn obstruct any $A_6$-equivariant branched cover $V_{\Ham} \to V_{\Kle}$. Note that we cannot use fixed points of the Sylow $3$-subgroup $P \leq A_6$ to derive a contradiction like we did for $V_{\Mod}$ --- there are no fixed points on $V_{\Ham} =\tau_{1,2,3}$. Instead, we will directly show that the Amitsur group of $V_{\Ham}$ is trivial.

\begin{proposition}\label{prop:HamiltonInvariantPicard}
    The $A_6$-invariant Picard group $\Pic(V_{\Ham})^{A_6}$ is generated by a primitive effective divisor, namely the restriction of the hyperplane class.
\end{proposition}
\begin{proof}
    Since $A_6$ acts symplectically on $V_{\Ham}$ by permuting coordinates, we can directly see that the $H^2(\tau_{1,2,3},\ZZ)^{A_6}$ is rank $3$ and generated by classes of Hodge type $(2,0)$, $(0,2)$, and $(1,1)$; here the $(1,1)$ class is generated by the primitive hyperplane class $H$ restricted to $\tau_{1,2,3}$ (see \cite[Table No. 79d]{brandhorst2021extensions} for the intersection form and more details for this projective K3 surface). Recall that the N\'eron--Severi group and Picard group of a K3 surface coincide, so the $A_6$-invariants coincide. Since $\Pic(V_{\Ham})^{A_6}$ is the orthogonal complement to the type $(2,0)$ and $(0,2)$ parts, we see that $\Pic(V_{\Ham})^{A_6} = \ZZ\langle H\rangle$.
\end{proof}

\begin{corollary}
    The Hamilton surface's Amitsur group is trivial, i.e. $\mathrm{Am}(V_{\Ham},A_6) = 0$.
\end{corollary}
\begin{proof}
    The line bundle corresponding to the restricted hyperplane class $H$ is $\mathcal{O}_{\tau_{1,2,3}}(1) = \mathcal{O}_{\PP^5}(1)|_{\tau_{1,2,3}}$, and hence linearizable. Thus the map $\varphi: \Pic(\widetilde{V}_{\Ham},A_6) \to \Pic(\widetilde{V}_{\Ham})^{A_6}$ is an isomorphism, proving the claim (the Amitsur group is the cokernel of the map $\varphi$).
\end{proof}

\begin{theorem}\label{NoKleinHamilton}
There is no $A_6$-equivariant generically finite dominant rational map $V_{\Ham} \to V_{\Kle}$.
\end{theorem}
\begin{proof}
    Just as in the proof of \Cref{NoModularKlein}, we proceed by contradiction. Since the supposed $A_6$-equivariant branched cover $V_{\Ham} \to V_{\Kle}$ would induce an injective map on Picard groups, we would have an inclusion on Amitsur groups 
    $$\ZZ/3\ZZ \cong \mathrm{Am}(V_{\Kle},A_6) \hookrightarrow \mathrm{Am}(V_{\Ham},A_6) = 0,$$ which is a contradiction. Thus no such equivariant branched cover could have existed.
\end{proof}

\begin{proof}[Proof of \Cref{MainTheorem}]
    This is just a synthesis of the results shown in \Cref{NoModularHamilton}, \Cref{NoModularKlein}, and \Cref{NoKleinHamilton}.
\end{proof}

\section{Solutions via Hilbert modular forms}\label{Section5}

What is presented in this section is implicit the literature in some form. As far as I am aware, the role that uniformization by automorphic forms plays in solutions to the general sextic is not emphasized anywhere in the modern literature except in \cite{farb2023modular}. Even so, no explicit analytic formula for the roots of a sextic is given in \textit{loc. cit.} --- in this sense, this section makes explicit the root finding method proposed in \cite[Proposition 4.14 (3)]{farb2023modular}.

\subsection{Reviewing the modular solution}

In order to explicitly solve for the roots of a general sextic, let us review the last step in the modular solution. At this point, we have used Tschirnhaus transformations to identify the square root of the discriminant $\Delta$ (so $\Delta^2$ is the discriminant), and to solvably reduce a general sextic to one of the form
\begin{equation}\label{ModularNormalForm}
    z^6 + az^3 +bz + c = 0.
\end{equation}
A Sage calculation shows that the ring of $A_6$-invariants on $V_{\Mod}$ is
$$\CC[\sigma _3 ,\sigma_5, \sigma_6 , \delta]/(\delta^2 - \Delta^2),$$
thus $V_{\Mod}/A_6$ is the hypersurface $Z(\delta^2 -\Delta^2) \subset \PP(3,5,6,15)$.

By the above discussion, the natural branched $A_6$-cover $\pi: V_{\Mod} \to V_{\Mod}/A_6$ is given by the morphism
$$\pi(p) = [\sigma_3(p):\sigma_5(p) : \sigma_6(p):\Delta(p)].$$
To explicitly solve for the roots of \Cref{ModularNormalForm}, we must build a (multi-valued) inverse of the function $\pi$. To do this, we will leverage the birational equivalence of $V_{\Mod}$ to a Hilbert modular surface.

\subsection{Uniformization by Hilbert modular forms}

A lattice $\Gamma \leq \Aut(\HH^2\times \HH^2)$ is called a \emph{Hilbert modular group}; the associated locally symmetric space $\Gamma \backslash \HH^2 \times \HH^2$ is called a \emph{Hilbert modular surface}. One of the main examples we will be interested in is the Hilbert modular group $\Gamma_5 = \SL_2(\mathcal{O}_{\QQ(\sqrt{5})})$ whose defining number field is $K=\QQ(\sqrt{5})$. The group $\Gamma_5$ acts holomorphically on $\HH^2 \times \HH^2$ by linear fractional transformations like so:
$$\begin{pmatrix*}
    a & b\\
    c & d
\end{pmatrix*} \cdot (z_1,z_2) = \left(\frac{az_1 +b}{cz_1+d}, \frac{a'z_2 +b'}{c'z_2+d'}\right),$$
where $(-)'$ denotes the Galois conjugate of the matrix entries in $\RR$.

Given a Hilbert modular surface $X= \Gamma \backslash \HH^2 \times \HH^2$ with defining number field $K$, the general theory of Baily--Borel guarantees that $X$ admits a compactification with the structure of a projective variety. Such a compactification is built by adding in finitely many \emph{cusps}, which are equivalence classes of points on $\PP^1(K)$ under the (extended) linear fractional $\Gamma$-action. Explicit projective embeddings of these varieties can be built using Hilbert modular forms, whose definition we now recall:

\begin{definition}
    Let $f: \HH^2\times \HH^2 \to \CC$ be a holomorphic function. We say that $f$ is a \emph{Hilbert modular form of weight $k$} on $\Gamma$ if $f$ satisfies the following transformation law under the $\Gamma$-action:
    $$f\left(\begin{pmatrix*}
    a & b\\
    c & d
\end{pmatrix*} \cdot (z_1,z_2)\right) = (cz_1+d)^k(c'z_2+d')^k f(z_1,z_2).$$
If $f$ vanishes at all of the cusps of $\Gamma \backslash \HH^2\times \HH^2$ (meaning a Fourier expansion of $f$ in a neighborhood of any cusp has vanishing constant term), we say that $f$ is a \emph{cusp form}. The space of weight $k$ cusp forms on $\Gamma$ is denoted by $S_k(\Gamma)$.
\end{definition}

Recall from Proposition \ref{vanderGeerBirational} that the normal form $V_{\Mod}=\tau_{1,2,4}$ is birational to the Hilbert modular surface $\Gamma_5[3]\backslash \HH^2\times \HH^2$, where
$$\Gamma_{5}[3] = \ker \left(\SL_{2}\left(\ZZ\left[\frac{1+\sqrt{5}}{2}\right]\right)\rightarrow \PSL_2(\FF_9)\right).$$
Note that $\Gamma_5 = \SL_2(\mathcal{O}_{\QQ(\sqrt{5})})$ is the source of this homomorphism. This quotient corresponds to the $A_6 \cong \PSL_2(\FF_9)$ congruence cover
\begin{equation}\label{CongruenceCover}
    \Gamma_5[3]\backslash \HH^2\times \HH^2 \to \Gamma_5\backslash \HH^2\times \HH^2.
\end{equation}
This cover is birationally equivalent to the $A_6$-covering that appears in the last step of the modular solution $\tau_{1,2,4} \to \tau_{1,2,4}/A_6$ \cite[Proposition 4.14 (3)]{farb2023modular}. Following the trail left by Green, it follows that the uniformization of the surface $V_{\Mod}$ by $\HH^2\times \HH^2$ would, in principle, offer solutions to generic sextics by Hilbert modular forms. We now make this relationship explicit.

Just as in \cite[Section 5]{green1978analytic}, the diagram we wish to fill in is as follows:

\begin{center}
    \begin{tikzcd}
\HH^2 \times\HH^2 \arrow[d,"\mu"] \arrow[rrd, "{(\psi_0,\dots,\psi_5)}"] \arrow[dd, bend right=49,"\Phi"']&                          &       \\
V_{\Mod} \arrow[d,"\pi"] \arrow[r,"="]                                      & {\tau_{1,2,4}} \arrow[r] & \PP^4\subset \PP^5\\
V_{\Mod}/A_6
\end{tikzcd}
\end{center}

Here the symbols $\psi_i$ denote the desired modular forms on $\HH^2\times \HH^2$ which will uniformize the normal form $V_{\Mod}=\tau_{1,2,4}$. These forms birationally identify $\Gamma_5[3]\backslash \HH^2 \times \HH^2$ with the variety $\tau_{1,2,4}$ in $\PP^5$ and will generate analytic solutions to the sextic. As discussed before, the canonical map lands in a hyperplane section $\PP^4\subset \PP^5$ determined by $\sigma_1$, and is given in terms of the holomorphic $2$-forms on $\tau_{1,2,4}$. In other words, this embedding is given by the canonical linear system on $V_{\Mod}$.

The complete intersection $V_{\Mod}=\tau_{1,2,4}$ embeds in $\PP^4\subset \PP^5$ by explicit equations, namely the common zero locus of the symmetric polynomials $\sigma_1$, $\sigma_2$, and $\sigma_4$. The image of the weight $2$ cusp forms $[s_0: \dots :s_5 ]: \HH^2 \times \HH^2 \rightarrow \PP^5$ on $\Gamma_5[3]$ lies within $\tau_{1,2,4}$, i.e. they satisfy the relations given by the same symmetric polynomials, and moreover they define a degree $1$ dominant morphism onto its image \cite[p. 193]{van1988hilbert}. Thus our desired Hilbert modular forms are $(\psi_0,\dots,\psi_5) = (s_0,\dots,s_5)$. The map 
\begin{align*}
    \varphi: S_2(\Gamma_5[3]) &\longrightarrow H^0(V_{\Mod},\Omega^2)\\
    f &\longmapsto fdz_1\wedge dz_2
\end{align*}
identifies the space of parallel weight $2$ cusp forms $S_2(\Gamma_5[3])$ with holomorphic $2$-forms $H^0(V_{\Mod},\Omega^2)$ \cite[Chapter III, Proposition 3.7]{van1988hilbert}. Once one calculates a basis $s_1',\dots,s_5'$ for the space of holomorphic $2$-forms on $V_{\Mod}$ via the Poincar\'e residue formula, and defines $s_0' = -\sum s_i'$, the corresponding cusp forms $s_i=\varphi^{-1}(s_i')$ will give an explicit analytic formula for points on the projectively embedded normal form $V_{\Mod} \subset \PP^5$.

To solve for the roots of a sextic in the form given by \Cref{ModularNormalForm}, we need only invert the projection $\pi$. However, we could also invert the full uniformization map $$\Phi = \pi \circ \mu: \HH^2 \times \HH^2 \to V_{\Mod}/A_6$$ to solve for the roots as well: given the (multi-valued) inverse function $\Phi^{-1}$ and the cusp forms $s_0\dots,s_5$, the roots of our sextic would be given by $s_i(\Phi^{-1}([a:b:c:\Delta]))$ for $i = 0,\dots,5$. We now focus our attention on explicitly inverting $\Phi: \HH^2 \times\HH^2 \to V_{\Mod}/A_6$.

\subsection{The uniformizing differential equation}
The congruence cover \ref{CongruenceCover} gives us an explicit quotient group structure to leverage when inverting $\Phi$ --- the map $\Phi$ is birationally equivalent to the uniformization $\HH^2 \times \HH^2 \to \Gamma_5 \backslash\HH^2 \times \HH^2$. Rationality of $\Gamma_5 \backslash\HH^2 \times \HH^2$, along with a basis for its function field in terms of modular forms, was worked out by Gundlach \cite{gundlach1963bestimmung}. In what follows, we denote the $\Gamma_5$-quotient by $X_5 = \Gamma_5 \backslash \HH^2 \times \HH^2$. The birational equivalence between $V_{\Mod}/A_6$ and $X_5$ is explicitly given by Hirzebruch's icosahedral model for $X_5$ (see \cite[Section 3]{hirzebruch2006ring}).

In the quintic case, Green produces an inverse $\Phi^{-1}$ by identifying it with the solution to the Schwarzian differential equation, the so-called \emph{uniformizing differential equation} \cite[Section 6]{green1978analytic}. For the case of sextics, we must consider the analogous uniformizing differential equation of Hilbert modular surfaces developed by Sasaki--Yoshida \cite{sasakiyoshidaI,sasakiyoshidaII} and Sato \cite{sato1991uniformizing}.

Identify $\HH^2 \times\HH^2$ with an open domain in a smooth quadric $\PP^1 \times \PP^1$ embedded in $\PP^3$. This allows us to assign the coordinates $(z_1,z_2) \in \HH^2 \times \HH^2$ to the projective coordinates $[1:z_1:z_2:z_1z_2] \in \PP^3$ satisfying the quadric condition. We pick local coordinates $(x,y)\in X_5$. Following \cite{sasakiyoshidaII}, the uniformizing equation of the Hilbert modular surface $X_5$ is given by the system of linear differential equations
\begin{align*}
    \frac{\partial^2 z}{\partial x^2} &= \ell \frac{\partial^2 z}{\partial x \partial y} + a\frac{\partial z}{\partial x } + b\frac{\partial z}{\partial y} + pz\\
    \frac{\partial^2 z}{\partial y^2} &= m \frac{\partial^2 z}{\partial x \partial y} + c\frac{\partial z}{\partial x } + d\frac{\partial z}{\partial y} + q z
\end{align*}
where $z$ is unknown and the set of coefficients $\{\ell,m,a,b,c,d,p,q\}$ are determined by the holomorphic structure on $X_5$ and various integrability conditions (the coefficients are explicitly given below). The space of solutions is dimension 4 (see \cite[Proposition 1.4]{sasakiyoshidaII}). A basis of solutions $z_0, \dots,z_3$ to the uniformizing differential equation will define the inverse $\Phi^{-1}$ like so:
\begin{align*}
    X_5 &\longrightarrow \HH^2\times \HH^2 \subset \PP^3\\
    (x,y) &\longmapsto [z_0(x,y):\dots:z_3(x,y)]
\end{align*}
Thus to produce the inverse function $\Phi^{-1}$, it suffices to determine the coefficients of the uniformizing differential equation associated to $X_5$. Thankfully, almost all of this work was accomplished by Sato \cite[Example 4]{sato1991uniformizing}, and corrected by Nagano \cite[Theorem 6.1, Remark 6.2]{nagano2012period}.

Let $\tau : (z_1,z_2) \mapsto (z_2,z_1)$ denote the involution of $\HH^2 \times \HH^2$ exchanging coordinates. The discrete group $\hat{\Gamma} _5= \langle \Gamma_5, \tau \rangle$ is also a Hilbert modular group. The associated quotient space, which we denote by $\hat{X}_5 = \hat{\Gamma}_5 \backslash \HH^2 \times \HH^2$, admits a degree $2$ orbifold covering by $X_5$, where the deck group action is induced by the coordinate exchange map $\langle \Bar{\tau}\rangle \cong \hat{\Gamma}_5 / \Gamma_5$. The space $\hat{X}_5$ was studied extensively by Gundlach \cite{gundlach1963bestimmung}, Hirzebruch \cite{hirzebruch2006ring}, M\"uller \cite{muller1983hilbertsche}, and others; as previously mentioned, the uniformizing differential equation was determined by Sato \cite{sato1991uniformizing} and Nagano \cite{nagano2012period}. We will mildly abuse notation and let $\Bar{\tau}$ denote the covering map
$$\Bar{\tau} : X_5 \to \hat{X}_5.$$
Due to work of Klein \cite{klein1884vorlesungen}, Hirzebruch \cite{hirzebruch2006ring}, and Kobayashi--Kushibiki--Naruki \cite{kobayashi1989polygons}, the Hilbert modular surface $\hat{X}_5$ is isomorphic to the weighted projective space $\PP(1,3,5)$ generated by Klein's icosahedral polynomials $\mathfrak{A}$, $\mathfrak{B}$, and $\mathfrak{C}$ of degrees $2$, $6$, and $10$, respectively (for a nice overview of these results, see \cite[pp. 820--821]{nagano2013theta}). We let $(x,y)$ denote the local coordinates on $\hat{X}_5$ where $\{\mathfrak{U}\neq 0\}$, i.e.
$$(x,y) = \left(\frac{\mathfrak{B}}{\mathfrak{A}^3},\frac{\mathfrak{C}}{\mathfrak{A}^5}\right).$$
Nagano calculated the (corrected) normalization factor of the differential equation associated to the Hilbert modular surface $\hat{X}_5$: it is given by the equation
$$e^{2\theta} = \frac{-36x^2 + 32x + y}{y^{1/2}(1728x^5 - 720x^3y + 80xy^2 - 64(5x^2 - y)^2 - y^3)^{3/2}}.$$
By construction, the normalization factor is invariant under the group $\hat{\Gamma}_5$, and therefore is also $\Gamma_5$-invariant. With this normalization factor, the coefficients of the uniformizing differential equation on $\hat{X}_5$ are given as follows (see \cite[Theorem 6.1, Remark 6.2]{nagano2012period}):
\begin{align*}
    \ell &= \frac{-20(4x^2 + 3xy - 4y)}{36x^2 - 32x - y}, \qquad m = \frac{-2(54x^3 - 50x^2 - 3xy + 2y)}{5y(36x^2 - 32x - y)},\\
    a &= \frac{-20(3x - 2)}{36x^2 - 32x - y}, \qquad b = \frac{-10(8x + 3y)}{36x^2 - 32x - y},\\ 
    c &= \frac{3x - 2}{5y(36x^2 - 32x - y)}, \qquad d = \frac{-198x
^2 + 180x + 7y}{5y(36x^2 - 32x - y)},\\
    p &= \frac{-3}{36x^2 - 32x - y}, \qquad q = \frac{3}{100y(36x^2 - 32x - y)}.
\end{align*}
Thus, up to the $\langle \tau \rangle$-orbit ambiguity in $\HH^2 \times \HH^2$, our desired (multi-valued) inverse map
$$\Phi^{-1}: X_5 \xrightarrow[]{\bar{\tau}} \hat{X}_5 \to \HH^2\times \HH^2$$
is given by a basis of solutions to the uniformizing differential equation of $\hat{X}_5$ with these prescribed coefficients.

To summarize, an analytic formula for the roots of the generic type $(1,2,4)$ reduced sextic $p(z)= z^6 +az^3 + bz + c $ is given by the following steps:
\begin{enumerate}
    \item Solve the uniformizing differential equation for the multi-valued inverse $\Phi^{-1}:X_5 \to \HH^2 \times \HH^2$ and write down the Hilbert modular cusp forms $s_0,s_1,\dots,s_5$ on $\Gamma_5[3]\backslash \HH^2 \times \HH^2$ using the Poincar\'e residue formula on $\tau_{1,2,4}$.
    \item Take the point $q= [a:b: c:\Delta]$ on $V_{\Mod}/A_6$ corresponding to the given sextic $p(z)$. Identify $q$ to the corresponding point $\bar{q}$ on the Hilbert modular surface $X_5$. Given a point in the image of the inverse map $\Phi^{-1}(\bar{q})$, evaluate it on each cusp form $s_i$ for $i = 0,\dots, 5$: these are the roots of the sextic $p(z)$.
\end{enumerate}
In a similar vein, one should be able to give an explicit analytic formula for the roots of a sextic in Hamilton's or Klein's normal form. This leads us to pose the following problem:

\begin{problem}
    Build a uniformization of Hamilton's normal form to solve for the roots of a general sextic using automorphic forms. Do the same for Klein's normal form.
\end{problem}

\printbibliography

\Addresses
\end{document}